\documentclass[reqno]{amsart}
\begin{document}

\newtheorem{conjecture}{Conjecture}{}
\newtheorem{question}{Question}{}
\newtheorem{theorem}{Theorem}[section]
\newtheorem{proposition}{Proposition}[section]
\newtheorem{lemma}{Lemma}[section]
\newtheorem{corollary}[theorem]{Corollary}
\newtheorem{example}[theorem]{Example}
\newtheorem{remark}{Remark}{}
\newtheorem{definition}{Definition}[section]
\setlength{\parskip}{.15in} 
\setlength{\parindent}{0in} 

\title[\hfil  On the irrationality of mock theta functions]{On the irrationality of Ramanujan's mock theta functions and other q-series at an infinite number of points } 
\author[Angelo B. Mingarelli \hfilneg]{Angelo B. Mingarelli}   
\address{School of Mathematics and Statistics\\ 
Carleton University, Ottawa, Ontario, Canada, K1S\, 5B6\\ and (present address) \\ Departamento de Matem\'aticas, Universidad de Las Palmas de Gran Canaria, 
Campus de Tafira Baja\\ 35017 Las Palmas de Gran Canaria, Spain.}
\email[A. B. Mingarelli]{amingare@math.carleton.ca, amingarelli@dma.ulpgc.es}

\date{{\rm Preprint}\, Dec. 24, 2007}
\thanks{}

\begin{abstract} We show that all of Ramanujan's mock theta functions of order 3, Watson's three additional mock theta functions of order 3, the Rogers-Ramanujan q-series, and 6 mock theta functions of order 5 take on irrational values at the points $q=\pm 1/2, \pm 1/3, \pm 1/4,\ldots.$
\end{abstract}

\maketitle

\section{Introduction} 

Let $q \in \mathbb{C}$, $|q|< 1$. Ramanujan's mock theta functions of order three are given by 
\begin{eqnarray}
\label{eq01} f(q) & := &1+ \sum_{n=1}^{\infty} \frac{q^{n^2}}{(1+q)^2(1+q^2)^2\cdots(1+q^n)^2}\\
\label{eq02} \phi (q) & := &1+ \sum_{n=1}^{\infty} \frac{q^{n^2}}{(1+q^2)(1+q^4)\cdots(1+q^{2n})},\\
\label{eq03} \psi (q) & := &\sum_{n=1}^{\infty} \frac{q^{n^2}}{(1-q)(1-q^3)\cdots(1-q^{2n-1})},\\
\label{eq04} \chi (q) & := &1+ \sum_{n=1}^{\infty} \frac{q^{n^2}}{(1-q+q^2)(1-q^2+q^4)\cdots(1-q^n +q^{2n})}.
\end{eqnarray}
In 1936 G.N. Watson \cite{wa} presented three more mock theta functions usually denoted by
\begin{eqnarray}
\label{eq05} \omega(q)& := & \sum_{n=0}^{\infty} \frac{q^{2n(n+1)}}{(1-q)^2(1-q^3)^2\cdots (1-q^{2n+1})^2},\\
\label{eq06} \nu (q) & := &\sum_{n=0}^{\infty} \frac{q^{n(n+1)}}{(1+q)(1+q^3)\cdots (1+q^{2n+1})},\\
\label{eq07} \rho (q) & := &\sum_{n=0}^{\infty} \frac{q^{2n(n+1)}}{(1+q+q^2)(1+q^3+q^6)\cdots (1+q^{2n+1}+q^{4n+2})}.
\end{eqnarray}
As a rule such functions are characterized by a power series expansion in $q = \exp{2 i\pi z}$ that converges in the unit disk. As functions of $z$ in the upper half plane, they admit an asymptotic expansion at the cusps, similar to those of modular forms of weight $1/2$, with possible poles at the cusps. They cannot be expressed as linear combinations of ordinary theta functions: this last statement is the most difficult to verify in general. For relationships between mock theta functions and modular forms we refer to recent work by Zweger \cite{z}. The notation is drawn from \cite{gr} and all the q-series represented are Ramanujan mock-theta functions of order three where the fairly standard terminology here is taken from \cite{ga},  \cite{dh} and \cite{wa}.

In this note we show that all of Ramanujan's mock theta functions of order three, $f$, $\phi$, $\psi$ and $\chi$ including Watson's additional contributions via the functions $\nu$, $\rho$ and $\omega$ take on irrational values at $q=\pm 1/2$, $\pm 1/3$, $\pm 1/4$,$\ldots.$ The same techniques are used to show that the Rogers-Ramanujan infinite series converge to irrational quantities at the same set of points. In addition, the following mock theta functions of order five denoted by $f_0$, $f_1$, $F_0$, $F_1$, $\Phi$, $\Psi$ share the same property with regards to irrationality at this set of points;
\begin{eqnarray}
\label{eq08} f_{0}(q) & :=&  \sum_{n=0}^{\infty} \frac{q^{n^2}}{(-q;q)_{n}},\\
\label{eq09} f_{1}(q)  & := & \sum_{n=0}^{\infty} \frac{q^{n(n+1)}}{(-q;q)_{n}}, \\
\label{eq010} F_{0}(q) & := & \sum_{n=0}^{\infty} \frac{q^{2n^2}}{(q;q^2)_{n}}, \\
\label{eq011} F_{1}(q)  & := & \sum_{n=0}^{\infty} \frac{q^{2n(n+1)}}{(q;q^2)_{n+1}}, \\
\label{eq012} \Phi(q)  & := & -1+ \sum_{n=0}^{\infty} \frac{q^{5n^2}}{(q;q^5)_{n+1}\,(q^4;q^5)_{n}}, \\
\label{eq013} \Psi(q)  & := & -1+ \sum_{n=0}^{\infty} \frac{q^{5n^2}}{(q^2;q^5)_{n+1}\,(q^3;q^5)_{n}}.
\end{eqnarray}
The irrationality of series defined by either Ramanujan's mock theta functions of order three or those found by Watson is derived by using results from the theory of Cantor series (described below). We make use of results in each of \cite{ht}, \cite{opp1} and an inversion in the unit disk to deduce this irrationality, although not for all (rational) values of $q$. Using these criteria we show that many of these functions are irrational at $q=\pm 1/2, \pm 1/3, \pm 1/4,\ldots.$ 

Finally we note that our methods fail to produce irrationality results for any other mock theta functions of order five let alone those of higher order. We observe that the proof of the mock theta conjectures for mock theta functions of order five by Hickerson \cite{dh} cannot be used to conclude the irrationality of the remaining ones as, in the end, we are dealing with sums and/or differences of irrational quantities.

\section{Irrationality of Mock Theta Functions of Order 3}

The theory of what are commonly called ``Cantor series" began in 1869 with Georg Cantor's second publication \cite{gc}, a seminal paper that presented a necessary and sufficient condition for series of the form 
\begin{equation}\label{cantor} S = \sum_{n=1}^{\infty} \frac{b_n}{a_1a_2\cdots a_n},\end{equation}
where the $a_i, b_i$ are integers to have irrational sums. The basic conditions being of the form $a_i \geq 2$, $a_i -1 \geq b_i \geq 0$, and for every integer $k > 0$ there is an $n$ such that $k | a_1a_2\cdots a_n$, Cantor showed that $S$ is irrational if and only if the $b_i > 0$ infinitely often and $a_i -1 > b_i$ infinitely often. This result received its first major extension only in 1955 as a result of the work of Oppenheim \cite{opp2},\cite{opp1}. The major development there was Oppenheim's dropping of the divisibility condition on the product of the first $n$ $a$-s and an extension of the theorem to the case where the $b_i$ can have both signs.  The next  development in the theory came recently in a paper by Han\v{c}l and Tijdeman \cite{ht} where the authors give different irrationality criteria which avoid the use of the Cantor-Oppenheim {\it a priori} condition $a_i -1 \geq b_i$.
These results are used to determine the irrationality of those mock theta functions mentioned in the introduction by performing an inversion in the unit disk. We prove the following results.

\begin{theorem} 
\label{th3-1} Ramanujan's mock theta functions of order three, \eqref{eq01}-\eqref{eq04}, take on irrational values at $q=\pm 1/2, \pm 1/3, \pm 1/4, \pm 1/5, \ldots.$ The same is true of Watson's mock theta functions of order three, \eqref{eq05}-\eqref{eq07}.
\end{theorem}

\begin{remark}{\rm  The proof of the irrationality of \eqref{eq01}-\eqref{eq04} above can be found embedded in {\rm [\cite{ma}, Theorem 1]} where, in addition, a measure of irrationality is given. Our proofs are much simpler though, require less machinery and can be used in the case of mock theta functions of order $5$ where nothing about irrationality appears to be known (see below).}
\end{remark}

\noindent{The same idea in the proofs give the following corollaries.}

\begin{theorem} The q-series of Rogers and Ramanujan {\rm [\cite{hw},\S~13.19], [\cite{gr},p.241]},
\begin{equation}\label{rr1} 1 + \sum_{n=1}^\infty \frac{q^{n^2}}{(1-q)(1-q^2)\cdots (1-q^n)},
\end{equation}
and 
\begin{equation}\label{rr2}1 + \sum_{n=1}^\infty \frac{q^{n(n+1)}}{(1-q)(1-q^2)\cdots (1-q^n)},
\end{equation}
converge to irrational quantities whenever $q=\pm 1/2, \pm 1/3, \pm 1/4,\ldots$.
\end{theorem}

\begin{corollary} For $q \in \mathbb{Z}$, $q \geq 2$, the infinite products
\begin{equation*} \prod_{m=0}^\infty \left (1-{1}/{q^{5m+1}} \right ) \left (1-{1}/{q^{5m+4}} \right ),\end{equation*}\begin{equation*} 
\prod_{m=0}^\infty \left (1-{(-1)^{m+1}}/{q^{5m+1}} \right ) \left (1-{(-1)^{m}}/{q^{5m+4}} \right ),
\end{equation*}
each converge to an irrational number. For $q \in \mathbb{Z}$, $q \leq -2$, the infinite products
\begin{equation*}\prod_{m=0}^\infty \left (1-{1}/{q^{5m+2}} \right ) \left (1-{1}/{q^{5m+3}} \right ),\end{equation*}\begin{equation*} 
\prod_{m=0}^\infty \left (1-{(-1)^{m}}/{q^{5m+2}} \right ) \left (1-{(-1)^{m+1}}/{q^{5m+3}} \right ),
\end{equation*}
each converge to an irrational number.
\end{corollary}

\begin{remark} {\rm A closely related problem to this one involving the Rogers-Ramanujan identities is considered independently in \cite{ma1}. It appears as if none of the methods in either  \cite{opp2}, \cite{ht} or \cite{opp1} can be used to deduce the irrationality of any of these functions at points other than those presented here. Basically, this is because if $p\neq 0, \pm 1$ in the fraction $p/q$, the numerators of the resulting Cantor series grow too rapidly for any of the tests to be of use. Still, as it stands, this technique should also prove  effective in treating further irrationality questions regarding such q-series. }
\end{remark}
\section{Irrationality of Mock Theta Functions of Order 5}
\begin{theorem} 
\label{th5-1}  The mock theta functions of order five \eqref{eq08}-\eqref{eq013} take on irrational values at $q=\pm 1/2, \pm 1/3, \pm 1/4,\ldots.$.
\end{theorem}

\begin{remark} {\rm We are unable to show that the other $6$ mock theta functions of order $5$, namely, $\phi_0(q)$, $\phi_1(q)$, $\psi_0(q)$, $\psi_1(q)$, $\chi_0(q)$, $\chi_1(q)$ are irrational at the specified points. To the best of our knowledge, none of the results in the literature on the irrationality of Cantor series appears to apply to these remaining $6$, thus more sensitive tests will be required to tackle these and the higher order mock theta functions.}
\end{remark}

\section{Proofs}

 In the following proofs, the symbols $m,n,k,N$ will always denote integers. First we need the following lemmas regarding the Cantor series $S$ defined at the outset. We adopt the notation in \cite{ht} for ease of exposition although one must note that the $a$'s and $b$'s are interchanged in \cite{ht},\cite{opp1}. For any $N \geq 1$,
\begin{eqnarray*}
\label{eq1} S_N &:= & \sum_{n=N}^{\infty} \frac{b_n}{a_Na_{N+1}\cdots a_n},
\end{eqnarray*}
where the $a_i,b_i$ are integers. The next lemmas are to be used interchangeably.

\begin{lemma}\label{lem2a}{\rm [\cite{opp1}, Theorem 4]} Let $(a_n), (b_n)$, be two sequences of integers with $a_n \geq 2$, $0 \leq b_n \leq a_n -1.$ If $b_n >0$ infinitely often and if there is a subsequence $i_n$ such that $a_{i_n}\to \infty$ and $b_{i_n}/a_{i_n}\to 0$ as $n \to \infty$, then $S$ as defined in \eqref{cantor} is irrational.
\end{lemma}

\begin{lemma}\label{lem2}{\rm [\cite{opp1}, Theorem 8]} Let $(a_n), (b_n)$, be two sequences of integers with $a_n\geq 2$, $ |b_n| \leq a_n -1.$ Furthermore, let $b_m\,b_n <0$ for some $m > i, n >i$ for any assigned integer $i$. If there is a subsequence $i_n$ such that $a_{i_n}\to \infty$ and $b_{i_n}/a_{i_n}\to 0$ as $n \to \infty$, then $S$ is irrational.
\end{lemma}

\begin{lemma}\label{lem1}{\rm [\cite{ht}, Proposition 3.1]} Let $(a_n)$, $a_n > 1$, $(b_n)$ be two sequences of integers satisfying $a_n\not|\ b_n$ for all $n$. If 
\begin{equation}
\label{eq2} \liminf_{N\to \infty} |S_N| =0,
\end{equation}
then $S$ is irrational. 
\end{lemma}

\begin{proof} (Theorem~\ref{th3-1})
\noindent{1)} Let $p, q \in \mathbb{C}$, $q\neq 0$. Then 
\begin{eqnarray}
f(p/q) & = & 1+\sum_{n=1}^{\infty} \frac{p^{n^2}\,q^{2\cdot(1+2+\ldots n)}}{q^{n^2}(p+q)^2(p^2+q^2)^2\cdots (p^n+q^n)^2},\nonumber \\
&=& 1+\sum_{n=1}^{\infty} \frac{p^{n^2}\,q^n}{(p+q)^2(p^2+q^2)^2\cdots (p^n+q^n)^2},\nonumber \\
& = & 1 + \frac{pq}{(p+q)^2}+\frac{pq}{(p+q)^2}\left \{\sum_{n=1}^{\infty} \frac{p^{n(n+2)}\,q^n}{(p^2+q^2)^2\ldots (p^{n+1}+q^{n+1})^2}\right \},\label{yo} \\
& := & 1 + \frac{pq}{(p+q)^2}+\frac{pq}{(p+q)^2} g(p,q) \nonumber
\end{eqnarray}
where $g$ is the function defined by the series on the right in \eqref{yo}. Observe that if $p, q \in \mathbb{Z}$ then, whenever it is defined, $f(p/q) \not\in \mathbb{Q}$ if and only if $g(p,q) \not\in \mathbb{Q}$. 

Consider the quantity $g(\pm 1, q)$ where $q \geq 2$ is an integer. Then (corresponding signs are to be used throughout)
$$
g(\pm 1, q)  = \sum_{n=1}^{\infty} \frac{(-1)^{n^2}\,q^n}{(q^2+1)^2(q^3\pm 1)^2(q^4+1)^2 \ldots (q^{n+1} \pm 1)^2}
$$
is of the form \eqref{cantor} with $b_n =(-1)^{n^2}\,q^n$ , $a_n = (q^{n+1}\pm 1)^2$, for $n \geq 1$. We verify the conditions of Lemma~\ref{lem2}: Since $q\geq 2$ is an integer, it is easy to see that $a_n > 2$, for all $n$ (in fact $a_n \geq 9$). Next the $b_n$ alternate in sign infinitely often so that the sign condition there is verified as well. In addition, a moment's notice shows that the  limiting behavior of the $a_n$ and the ratio $b_n/a_n$ is as required. The final condition, i.e., $a_n -|b_n| - 1\geq 0$ is equivalent to showing that $q^{2n+2}\pm 2q^{n+1} -q^n \geq 0$ for any $q \geq 2$ and $n \geq 1$. Since $q^{n+2} \geq 2q+1$ for any $n\geq 1$ and given $q \geq 2$ this last condition is also satisfied. The Lemma therefore gives us the irrationality of the quantity $g(\pm 1, q)$ for any integer $q \geq 2$, and this completes the proof.

Similar proofs apply to the remaining cases so we need only sketch the details.

\noindent{2)} The proof is similar to that for $f(q)$. Given $p$, $q\neq 0$,
$$\phi (p/q) = 1+\sum_{n=1}^{\infty} \frac{p^{n^2}\,q^n}{(p^2+q^2)(p^4+q^4)\cdots(p^{2n}+q^{2n})},$$
so that by setting $p=\pm 1$ we get, 
$$\phi (\pm 1/q) = 1+\sum_{n=1}^{\infty} \frac{(-1)^{n^2}\,q^n}{(q^2+1)(q^4+1)\cdots(q^{2n}+1)},$$
Note that this is a Cantor series \eqref{cantor} with $a_n=q^{2n}+1$ and $b_n=(-1)^{n^2}q^{n}$, $n \geq 1$. All the conditions of Lemma~\ref{lem2} being satisfied we conclude that $\phi (\pm 1/q)$ is irrational for all integers $q \geq 2$.

\noindent{3)} For given $p, q$ for which the series is defined,
$$\psi (p/q) = \sum_{n=1}^{\infty} \frac{p^{n^2}}{(q-p)(q^3-p^3)\cdots(q^{2n-1}-p^{2n-1})}.$$ From this we find that
$$\psi (\pm 1/q) = \sum_{n=1}^{\infty} \frac{(\pm 1)^{n^2}}{(q\mp 1)(q^3\mp 1)\cdots(q^{2n-1}\mp 1)}$$ 
holds for any integer $q\geq 2$. We rewrite the preceding series in the form
$$\psi (\pm 1/q) = \frac{{\pm 1}}{q\mp 1}\ \,\left \{ 1+  \sum_{n=2}^{\infty} \frac{(\pm 1)^{n^2-1}}{(q^3\mp 1)(q^5\mp 1) \cdots(q^{2n-1}\mp 1)}\right \} $$ and view the series on the right above as a Cantor series with $b_n=(\pm 1)^{n^2-1}$, $a_n = q^{2n-1}\mp 1,$ for $n \geq 2$. Once again it is not difficult to see that the conditions of Lemma~\ref{lem2} are satisfied so that the series on the right of the last display converges to an irrational number and thus, so does $\psi (\pm 1/q)$.

\noindent{4)} In this case we have $$\chi (p/q) = 1+ \sum_{n=1}^{\infty} \frac{p^{n^2}}{(q^2-pq+p^2)(q^4-p^2q^2+p^4)\cdots (q^{2n}-p^nq^n+p^{2n})},$$  for any $p, q$ for which the series converges. Setting $p=\pm 1$, $q \geq 2$ an integer, we find the form
$$\chi (\pm 1/q) = 1+ \sum_{n=1}^{\infty} \frac{(-1)^{n^2}}{(q^2\mp q+1)(q^4-q^2+1)\cdots (q^{2n} - (\pm 1)^n q^n+1)}.$$ Note that $a_n = q^{2n} - (\pm 1)^n q^n+1 \geq 2$ for every $n \geq 1$ and $q \geq 2$ an integer. As before we set $b_n = (\pm 1)^{n^2}$, for $n \geq 1$. The remaining conditions of the Lemma are readily verified.

\noindent{5)} As before, for $p, q$ for which the series is defined, a simple change of variable gives the form
$$\omega (p/q) =\sum_{n=1}^{\infty} \frac{p^{2n(n-1)}q^{2n}}{(q-p)^2(q^3-p^3)^2\cdots (q^{2n-1}-p^{2n-1})^2}.$$
Hence, for $q\geq 2$ an integer, we find
\begin{eqnarray*}
\omega (\pm 1/q) &=& \sum_{n=1}^{\infty} \frac{q^{2n}}{(q\mp 1)^2(q^3\mp 1)^2\cdots (q^{2n-1}\mp 1)^2}\\
& = & \frac{q^2}{(q\mp 1)^2} \, \left \{ 1 +   \sum_{n=1}^{\infty} \frac{q^{2n}}{(q^3\mp 1)^2\cdots (q^{2n+1}\mp 1)^2}             \right \}.
\end{eqnarray*}
It follows that $\omega (\pm 1/q) $ is irrational if and only if the series on the right of the above display is irrational. But a simple substitution as is by now common, namely, $a_n=(q^{2n+1}\mp 1)^2$, $b_n=q^{2n}$ shows that for $q\geq 2$ and every $n\geq 1$ all the conditions of Lemma~\ref{lem2a} are satisfied. There follows the irrationality of said series as well as the quantity $\omega (\pm 1/q) $ for $q \geq 2$ and this completes the proof.
\end{proof}

\noindent{ 6) and 7)} Arguments similar to the ones presented in the proofs above apply to the case of the remaining two Watson mock theta functions $\nu (q)$ and $\rho (q)$ to show irrationality for $q=\pm 1/2,\pm 1/3,\pm 1/4,\ldots$ (so we leave the proofs to the reader).

\begin{remark}{\rm Alternate proofs of these results can also be obtained using Lemma~\ref{lem1} (or [\cite{ht}, Proposition~3.1]). For example, we show how (1) can be so proved: that is e.g., $f(1/q)$ is irrational for all integers $ q \geq 2$. Now
$$f(1/q) = 1+\sum_{n=1}^{\infty} \frac{q^n}{(1+q)^2(1+q^2)^2\cdots(1+q^n)^2}.$$
Define $a_n=(1+q^n)^2$ for $n \geq 1$, $b_n=q^{n}$, for all $n \geq 1$. Fix $q \in \mathbb{Z^+}$, $q \geq 2$. Then $a_n > 1$, for all $n \geq 1$. Observe that since $q^n | q^{n^2}$ for $n \geq 1$ then $a_n\not|\ b_n$ for such $n$. In addition, for any $N \geq 1$
$$\prod_{k=N}^n{(1+q^k)^2} > q^{2N+2(N+1)+\ldots+2n} = q^{n^2+n -(N^2-N)}.$$
Since $a_Na_{N+1}\cdots a_n \geq q^{n^2+n -(N^2-N)}$, we have
\begin{eqnarray*}
S_N & =&\sum_{n=N}^{\infty} \frac{b_n}{a_Na_{N+1}\cdots a_n}\\
& \leq & \sum_{n=N}^{\infty} q^n/q^{n^2+n -(N^2-N)} =  \sum_{n=N}^{\infty} 1/q^{n^2 -(N^2-N)} \\
& \leq & \frac{1}{q^N}\, \sum_{m=0}^{\infty} 1/q^{m(m+2N)} \leq \frac{1}{q^N}\, \sum_{m=0}^{\infty} 1/q^{m^2}\\
&=& o(1), \,{\rm as\,\,} N \to \infty,
\end{eqnarray*}
since $q \geq 2$. The result follows on account of Lemma~\ref{lem1}.

}
\end{remark}

\begin{proof} (Theorem 1.3)
The proofs are all too familiar by now. As usual replace $q$ by $p/q$ in \eqref{rr1} and rearrange the terms in the form of a Cantor series. Call the left-side of \eqref{rr1}, $r_1(q)$. Then for appropriate $p,q$, $r_1(p/q)$ may be written in the form
\begin{equation*}
r_1(p/q)  =  1+ \sum_{n=1}^\infty \frac{p^{n^2}\,q^n}{[q(q-p)][q^2(q^2-p^2)]\cdots [q^n(q^n-p^n)]}.
\end{equation*}
\end{proof}
Now set $p=\pm 1$, $q\in \mathbb{Z}$, $q \geq 2$. Note that the coefficients defined by $a_n = q^n(q^n - (-1)^n)$, $b_n = (-1)^{n^2}\,q^n$ satisfy all the conditions of Lemma~\ref{lem2}, thus the resulting Cantor series sums to an irrational number.

For the next one, write $r_2(q)$ for the left side of \eqref{rr2}. Then $r_2(p/q)$ can be written in the form
$$r_2(p/q)  =  1+ \sum_{n=1}^\infty \frac{p^{n(n+1)}}{[q(q-p)][q^2(q^2-p^2)]\cdots [q^n(q^n-p^n)]}.$$
Substitute $p=\pm 1$. Then the previous series becomes simply
$$r_2(p/q)  =  1+ \sum_{n=1}^\infty \frac{1}{[q(q\mp 1)][q^2(q^2-1)]\cdots [q^n(q^n-(-1)^n)]}.$$
This time we use $a_n = q^n(q^n - (-1)^n)$, $b_n = 1$ and Lemma~\ref{lem2a}. The conclusion follows.
\begin{proof}(Corollary 1.4)  This follows by the Rogers-Ramanujan identities applied to $r_i(\pm 1/q)$, cf., [\cite{hw},Theorem 362], [\cite{gr},p.241].
\end{proof}

\begin{proof} (Theorem~\ref{th5-1})
\vskip0.15in
\noindent{1)} Expanding the Pochhammer symbol observe that $$f_{0}(q) = 1 +\sum_{n=1}^{\infty} \frac{q^{n^2}}{(1+q)(1+q^2)\cdots (1+q^n)},$$ so that the inversion $q\to 1/q$ gives the form $$f_{0}(1/q) = 1+\sum_{n=1}^{\infty} \frac{q^{- n(n-1)/2}}{(1+q)(1+q^2)\cdots (1+q^n)}$$
after simplification. Set $b_n = q^{- n(n-1)/2}$, $a_n = 1+q^n$ so that the series is of the form \eqref{cantor}. For $q \in \mathbb{Z^+}$, $q \geq 2$, note that for every $n \geq 1$, $a_n > 2$, $a_n \geq b_n+1$, $a_n\to \infty$ and $b_n/a_n \to 0$ as $n \to \infty$. The stated irrationality now follows by Lemma~\ref{lem2a}.

The case $q\in \mathbb{Z}$, $q \leq -2$ may be treated similarly by using the device whereby for $q > 0$ we can rewrite the series in the form 
\begin{eqnarray*}
f_{0}(-1/q) &=& 1+\sum_{n=1}^{\infty} \frac{(-1)^{n^2}\, q^{- n(n-1)/2}}{(q-1)(q^2+1)\cdots (q^n+(-1)^n)}\\
&=& 1-\frac{1}{q-1} + \sum_{n=2}^{\infty} \frac{(-1)^{n^2}\, q^{- n(n-1)/2}}{(q^2+1)\cdots (q^n+(-1)^n)}\\
& = & 1-\frac{1}{q-1} + \sum_{n=1}^{\infty} \frac{(-1)^{n+1}\, q^{- n(n+1)/2}}{(q^2+1)\cdots (q^{n+1}+(-1)^{n+1})}
\end{eqnarray*}
where now the series on the right is a Cantor series with $b_n=(-1)^{n+1}\, q^{- n(n+1)/2}$, $a_n=q^{n+1}+(-1)^{n+1}$ and these terms satisfy the conditions of Lemma~\ref{lem2}, for $q\geq 2$ an integer. Since $q\geq 2$ is an integer and the series is irrational, the result follows. 

\noindent{2)} As in (1) we expand the symbol and perform the inversion $q\to 1/q$ as before to find the form
$$f_{1}(1/q) = 1+\sum_{n=1}^{\infty} \frac{q^{- n(n+1)/2}}{(1+q)(1+q^2)\cdots (1+q^n)}$$
after simplification. The proof now follows the lines of the previous one with minor changes and so is omitted. 

\noindent{3)} By definition, 
\begin{eqnarray*}
F_{0}(q) &=& 1+\sum_{n=1}^{\infty} \frac{q^{2n^2}}{(1-q)(1-q^3)\cdots (1-q^{2n-1})}
\end{eqnarray*}
so, for $p, q \in \mathbb{C}$, $q \neq 0$ we have, after some simplification,
\begin{eqnarray*}
F_{0}(p/q) &=& 1+\sum_{n=1}^{\infty} \frac{p^{2n^2}}{q^{n^2}(q-p)(q^3-p^3)\cdots (q^{2n-1}-p^{2n-1})}
\end{eqnarray*}
Inserting $p=\pm 1$ in the preceding expression gives
\begin{eqnarray}\label{eqf0q}
F_{0}(\pm 1 /q) &=& 1+\sum_{n=1}^{\infty} \frac{1}{q^{n^2}(q\mp 1)(q^3\mp 1)\cdots (q^{2n-1}\mp 1)}.
\end{eqnarray}
Note that for every $q \geq 2$, $q \in \mathbb{Z}$ the right side of \eqref{eqf0q} is a Cantor series with the identifications $a_n=q^{2n-1}(q^{2n-1}\mp 1)$ and $b_n=1$ for all $n$. In any case, for every $q \geq 2$ an integer and any $n \geq 1$, $a_n \geq 2$ and $a_n\to \infty$. In addition, $b_n/a_n \to 0$ as $n \to \infty$, and so by Lemma~\ref{lem2a} $F_0(\pm 1/q)$ is irrational for every integer $q\geq 2$.

\noindent{4)} In this case, 
\begin{eqnarray*}
F_{1}(q) &=& \sum_{n=0}^{\infty} \frac{q^{2n^2+2n}}{(1-q)(1-q^3)\cdots (1-q^{2n+1})}.
\end{eqnarray*}
As before, for $p, q \in \mathbb{C}$, $q \neq 0$ we get
\begin{eqnarray*}
F_{1}(p/q) &=&\frac{q}{q-p}+ \sum_{n=1}^{\infty} \frac{p^{2n^2+2n}}{q^{n^2-1}(q-p)(q^3-p^3)\cdots (q^{2n+1}-p^{2n+1})}.
\end{eqnarray*}
The substitution $p=\pm 1$ now produces
\begin{eqnarray}\label{eqf1q}
F_{1}(\pm 1 /q) &=& \frac{q}{q\mp 1}+ \frac{1}{q\mp 1}\, \sum_{n=1}^{\infty} \frac{1}{q^{n^2-1}(q^3\mp 1)\cdots (q^{2n+1}\mp 1)} \\
& = & \frac{q}{q\mp 1}+ \frac{1}{q\mp 1}\, \sum_{n=1}^{\infty} \frac{q}{q^{1+ 2+ \cdots + (2n-1)}(q^3\mp 1)\cdots (q^{2n+1}\mp 1)}.\label{eqf1q1}
\end{eqnarray}
Note that since $q\geq 2$ is an integer, $F_{1}(\pm 1 /q)$ is irrational if and only if the series on the right of \eqref{eqf1q} is irrational. But, observe that for such $q$ the right side of \eqref{eqf1q1} is a Cantor series with the identifications $a_n=q^{2n-1}(q^{2n+1}\mp 1)$ and $b_n=q$ for all $n$. Also, for every $q \geq 2$ an integer and any $n \geq 1$, $a_n \geq 2$ and $a_n\to \infty$. In addition, $b_n/a_n \to 0$ as $n \to \infty$ therefore, by Lemma~\ref{lem2a}, $F_1(\pm 1/q)$ is also irrational for every integer $q\geq 2$.

\noindent{5)} Expanding the Pochhammer symbol we find
\begin{eqnarray*}
\Phi (q) &=& -1 + \sum_{n=0}^{\infty} \frac{q^{5n^2}}{(1-q)(1-q^4)(1-q^6)(1-q^9)\cdots (1-q^{5n-1})(1-q^{5n+1})}\\
& = &  -1 +\frac{1}{1-q}+ \frac{1}{1-q}\,\sum_{n=1}^{\infty} \frac{q^{5n^2}}{(1-q^4)(1-q^6)\cdots (1-q^{5n-1})(1-q^{5n+1})}.
\end{eqnarray*}
If $p, q \in \mathbb{C}$, $q \neq 0$ we see that
\begin{eqnarray}
\Phi (p/q) &=&-1 + \frac{q}{q-p}+ \frac{q}{q-p}\, \sum_{n=1}^{\infty} \frac{p^{5n^2}\,q^{4+6+9+11+\cdots+(5n-1)+(5n+1)}}{q^{5n^2}(q^4-p^4)(q^6-p^6)\cdots (q^{5n+1}-p^{5n+1})}\nonumber \\
&=& \ldots + \frac{q}{q-p}\,\sum_{n=1}^{\infty} \frac{p^{5n^2}\,q^{5n^2+5n}}{q^{5n^2}(q^4-p^4)(q^6-p^6)\cdots (q^{5n+1}-p^{5n+1})}\nonumber \\
&=& \ldots + \frac{q}{q-p}\,\sum_{n=1}^{\infty} \frac{p^{5n^2}\,q^{5n}}{(q^4-p^4)(q^6-p^6)\cdots (q^{5n+1}-p^{5n+1})}.\label{eqphi}
\end{eqnarray}
It clearly suffices to show that the series \eqref{eqphi} in question is irrational for $p=\pm 1$, $q\geq 2$ an integer. In other words, since $(-1)^{5n^2}=(-1)^n$ for every $n$, it suffices to show that 
\begin{eqnarray*}
\Phi (\pm 1/q) =... +\frac{q}{q\mp 1}\,\sum_{n=1}^{\infty} \frac{(-1)^n\,q^{5n}}{(q^4-1)(q^6-1)(q^9\mp 1)(q^{11}\mp 1)\cdots (q^{5n+1}-(\pm 1)^{5n+1})}
\end{eqnarray*}
is irrational (since the omitted expression is a rational number). However, the preceding series is a Cantor series as can be ascertained by defining $$a_n=(q^{5n-1} - (\pm 1)^{5n-1})(q^{5n+1}-(\pm 1)^{5n-1})$$ and $b_n=(-1)^n\,q^{5n}$ for all $n\geq 1$. Indeed, for every integer $q \geq 2$ and any $n \geq 1$, $a_n \geq 2$ and $a_n\to \infty$. In addition, $|b_n| \leq a_n -1$ for all $n \geq 1$, $q \geq 2$ and furthermore, $b_n/a_n \to 0$ as $n \to \infty$. Therefore, by Lemma~\ref{lem2a} (resp. Lemma~\ref{lem2}), $\Phi (\pm 1/q)$ is irrational for every integer $q\geq 2$.

\noindent{6)} This case is similar to the preceding one, so we need only sketch the details. Expanding the Pochhammer symbol we find
\begin{eqnarray*}
\Psi (q) &=& -1 + \sum_{n=0}^{\infty} \frac{q^{5n^2}}{(1-q^2)(1-q^3)(1-q^7)(1-q^8)\cdots (1-q^{5n-2})(1-q^{5n+2})}\\
& = & -1 +\frac{1}{1-q^2}+ \frac{1}{1-q^2}\,\sum_{n=1}^{\infty} \frac{q^{5n^2}}{(1-q^3)(1-q^7)\cdots (1-q^{5n-2})(1-q^{5n+2})}.
\end{eqnarray*}
Whenever $p, q \in \mathbb{C}$, $q \neq 0$,
\begin{eqnarray}
\Psi (p/q) =-1 + \frac{q^2}{q^2-p^2}+ \frac{q^2}{q^2-p^2}\, \sum_{n=1}^{\infty} \frac{p^{5n^2}\,q^{5n}}{(q^3-p^3)(q^7-p^7)\cdots (q^{5n+2}-p^{5n+2})}.
\label{eqpsi}
\end{eqnarray}
As before we note that 
\begin{eqnarray*}
\Psi (\pm 1/q) =... + \frac{q^2}{q^2-1}\,\sum_{n=1}^{\infty} \frac{(\pm 1)^n\,q^{5n}}{(q^3\mp 1)(q^7\mp 1)\cdots (q^{5n+2}-(\pm 1)^{5n+2})},
\end{eqnarray*}
where the omitted expression is a rational number for $q\geq 2$, an integer. Defining $$a_n=(q^{5n-2} - (\pm 1)^{5n-2})(q^{5n+2}-(\pm 1)^{5n+2})$$ and $b_n=(\pm 1)^n\,q^{5n}$ for all $n\geq 1$, we see that the previous series is a Cantor series. For every integer $q \geq 2$ and any $n \geq 1$, $a_n \geq 2$ once again, and $a_n\to \infty$. In addition, $|b_n| \leq a_n -1$ for all $n \geq 1$, $q \geq 2$ and $b_n/a_n \to 0$ as $n \to \infty$. Therefore, by Lemma~\ref{lem2a} (resp. Lemma~\ref{lem2}), $\Psi (\pm 1/q)$ is irrational for every integer $q\geq 2$.

\section{Acknowledgments}

The author is grateful to Matala-Aho for a fruitful correspondence and for reprints and preprints of his papers. My thanks also to Kathrin Bringmann for helpful comments and suggestions.

\end{proof}

\end{document}